\documentclass[14pt]{article}
\usepackage{geometry}[1inch]
\usepackage{amsthm,amsmath,amsfonts,amssymb, mathrsfs}
\usepackage{graphicx}
\usepackage[colorlinks,citecolor=blue,urlcolor=blue]{hyperref}
\theoremstyle{plain}   
\newtheorem{theorem}{Theorem}[section]
\newtheorem{lemma}[theorem]{Lemma}

\newtheorem{condition}[theorem]{Condition}
\newtheorem{definition}[theorem]{Definition}
\newtheorem{remark}[theorem]{Remark}
\newtheorem{assumption}[theorem]{Assumption}
\newtheorem{example}[theorem]{Example}

\newcommand{\R}{\mathbb{R}}

\newcommand{\di}{\mathrm{d}}

\newcommand{\E}{\mathbb{E}}

\newcommand{\Pb}{\mathbb{P}}

\begin{document}

\title{Wong--Zakai approximation of regime-switching SDEs via rough path theory}
\author{Jasper Barr\textsuperscript{1}\footnote{Corresponding author: jasper.barr@anu.edu.au}, Giang T. Nguyen\textsuperscript{2}, Oscar Peralta\textsuperscript{3}\\[\bigskipamount] \small \textsuperscript{1}Mathematical Sciences Institute, Australian National University \\ \small\textsuperscript{2}School of Mathematical Sciences, University of Adelaide \\\small\textsuperscript{3}School of Operations Research and Information Engineering, Cornell University}

\date{}



\maketitle
\begin{abstract}
This paper investigates the convergence of Wong--Zakai approximations to regime-switching stochastic differential equations, generated by a collection of finite-variation approximations to Brownian motion. We extend the results of \cite{nguyen2021rate} to $\R^d$-valued RSSDE by utilising rough path theoretic tools, acquiring the same modification of rate. \\


\end{abstract}





\section{Introduction}
    The ability to incorporate uncertainty has become a standard aspect of modern applied models. Indeed, the specific class of models known as stochastic differential equations (SDE) has found a wide variety of applications in modelling continuous phenomena across fields such as quantitative finance, ecology, mathematical biology, dynamical systems, and so on. However, SDEs suffer in that they only account for \textit{continuous} sources of uncertainty, typically driven by an $\mathbb{R}^d$-valued Brownian motion $B$. This deficiency motivates the notion of \textit{regime-switching}, in which the evolution of the stochastic process $Y$ depends on an additional jump process $J$ with countable state space $\mathcal{M}$:
    \begin{equation*}
        \di Y_t = \mu_{J_t}(t, Y_t) \, \di t + \sigma_{J_t}(t, Y_t) \, \di B_t\,,
    \end{equation*}
    where $\{\mu_i \}_{i \in \mathcal{M}}$ and $\{\sigma_i \}_{i \in \mathcal{M}}$ are a collection of drift and diffusion vector fields indexed by the (countable) state space $\mathcal{M}$ of the jump process $J$. Regime-switching SDEs (RSSDEs) allow one to account for discrete changes in environment that influence the continuous dynamics of $Y$. Think of a pressure sensor which is either operational or error-prone, or the trajectory of a stock price in a ``good'' or ``bad'' market.

    Regime-switching models have found a number of applications in quantitative finance \cite{duan2002option-pricing, zou2013}, in modelling rough volatility \cite{yin2010hybrid_switching}, and more recently in mean-field games \cite{bensoussan2020}.

    Naturally, the introduction of the jump process complicates the process of approximating such processes. Recently the first paper dealing with so-called Wong--Zakai approximations of 1-dimensional RSSDEs emerged \cite{nguyen2021rate}. The fundamental question addressed there is whether the strong convergence of finite-variation approximations $\{B^\lambda\}_{\lambda \geq 0}$ to a Brownian motion passes to convergence of the approximating RSSDEs $\{Y^\lambda\}_{\lambda \geq 0}$ to $Y$, and how the rate of convergence is affected. Recall that a collection of stochastic processes $\{B^\lambda\}_{\lambda \geq 0}$ is said to converge strongly to $B$ with rate function $\delta: \R^+ \rightarrow \R^+$ if for all $r$
    \begin{equation}\label{eq:sc-sup}
        \Pb \big( \Vert B - B^\lambda\Vert_{\infty; [0,T]} \geq k\delta(\lambda)\big) = o(\lambda^{-r})\,,
    \end{equation}
    where $k=k(r, T)$ is a constant dependent only on $r$ and $T$, and $\Vert \cdot\Vert_{\infty;[0,T]}$ denotes the supremum norm over the interval $[0,T]$. To address this question, Nguyen and Peralta \cite{nguyen2021rate} utilised a local pathwise estimate of the form 
    \begin{equation}\label{eq:lip-estimate-sup}
        \Vert Y - Y^\lambda\Vert_{\infty; [0,T]} \leq L \Vert B - B^\lambda \Vert_{\infty; [0,T]}\,,
    \end{equation}
    where $L$ is a random variable with desirable probabilistic properties. The estimate \eqref{eq:lip-estimate-sup} allows one to locally control the error of the regime-switching SDE approximation by the error in the driving processes, while the behaviour of $L$ is exploited to extend the strong convergence of $B^\lambda \rightarrow B$ to $Y^\lambda \rightarrow Y$ with a slightly worse rate, in the sense that
    \begin{equation*}
        \Pb \left( \Vert Y-Y^\lambda\Vert_{\infty;[0,T]} \geq \beta \delta(\lambda)\lambda^\varepsilon \right) = o(\lambda^{-r})
    \end{equation*}
    for all $\varepsilon >0$ and $r>0$.

    Unfortunately, \eqref{eq:lip-estimate-sup} fails to hold in the multi-dimensional case, due to the fact that the topology induced by the supremum norm is ill-suited to providing such estimates in dimension $d>1$. In the higher-dimensional setting, one can introduce oscillations at small scales to the approximations $\{B^\lambda\}$ such that the limit $\Vert B-B^\lambda \Vert_{\infty;[0,T]} \rightarrow 0$ remains unaffected, but it no longer holds that $\Vert Y - Y^\lambda \Vert_{\infty; [0,T]} \rightarrow 0$. The approximations $\{Y^\lambda \}_{\lambda \geq 0}$ can in fact converge to a \textit{different} limiting object given by $Y$, plus some correction term that vanishes in the scalar case or when vector fields built from $\{\sigma_i\}_{i \in \mathcal{M}}$ commute (see \cite[Theorem 7.2]{ikeda1992sde-diffusion} for further discussion). 
    
    To resolve this problem we utilise \textit{rough path theory}, introduced by Lyons \cite{lyons1998rough_paths} in the mid 90s. Rough path theory avoids the aforementioned problem by working with a stronger topology, induced by the so-called \textit{inhomogeneous rough path metric} $\rho_{p-var}$. This metric is built from a collection of path-norms, one of which is the $p$-variation seminorm of a path $X:[0,T] \rightarrow \mathbb{R}^d$
    \begin{equation}\label{eq:p-var}
    \Vert X \Vert_{p-var;[0,T]}^p := \sup_{\mathcal{P} \subset [0,T]} \sum_{t_i \in \mathcal{P}} |X_{t_i+1} - X_{t_i}|^p \,, \qquad p\in [1,\infty),
    \end{equation}
    where $|\cdot|$ denotes the Euclidean norm and the supremum is taken over all partitions $\mathcal{P}$ of $[0,T]$. Convergence of $B^{\lambda}$ to $B$ in $p$-variation norm not only implies convergence in supremum norm, but it also tracks small-scale oscillations. The payoff is that the local Lipschitz continuity of the solution map is restored, albeit in the different function space with a stronger topology that requires ``higher order'' information than just the path $B$. We discuss this in further detail in Section \ref{sec:preliminaries}.

    To run a similar argument as in \cite{nguyen2021rate}, we utilise a recent refinement \cite{2013friz:improved-lipschitz-estimate} of the Lipschitz estimate that applies to a class of Gaussian processes (and approximations thereof), of which Brownian motion is included. This leads us to the main result:

    \begin{theorem}
        Under the conditions of Theorem \ref{thm:sup-estimate} and Assumption \ref{asm:jump-tails}, suppose that there exists $\delta: \mathbb{R}_+ \rightarrow \mathbb{R}_+$ with $\lim_{\lambda \rightarrow \infty} \delta(\lambda) = 0$ such that for all $r>0$ 
\begin{equation*}
    \mathbb{P}\big( \rho_{p-var;[0,T]}(\mathbf{B}^\lambda, \mathbf{B}) \geq k\delta(\lambda)\big) = o(\lambda^{-r})\,,
\end{equation*}
where $k=k(r,T)>0$ is a constant dependent on $r$ and $T$ only. Then there exists some constant $\beta =\beta(r,T)>0$ such that for all $\varepsilon >0$ 
\begin{equation}
    \mathbb{P}\left(\big\Vert Y- Y^\lambda\big\Vert_{\infty; [0,T]} \geq \beta \delta(\lambda) \lambda^\varepsilon \right) = o(\lambda^{-r}).
\end{equation}
    \end{theorem}

    Assumption \ref{asm:jump-tails} matches Assumption 4 of \cite{nguyen2021rate}, which ensures that the jumps of $J$ are well-behaved by imposing bounds on the tail probability of the number of jumps over the compact interval $[0,T]$. As noted in \cite{nguyen2021rate}, this assumption is not overly restrictive in that it allows for deterministic, time-homogeneous Markov, time-inhomogeneous Markov, and semi-Markov processes. The conditions of Theorem \ref{thm:sup-estimate} are those required to guarantee the existence of a unique solution to the corresponding \textit{rough} regime-switching equation associated to the regime-switching SDE of interest. Notably, this requires much higher regularity than the standard assumption of local Lipschitz continuity and linear growth from SDE theory.

   The structure of the paper is as follows. We briefly recall in Section \ref{sec:preliminaries} some main results and definitions from rough path theory. In Section \ref{sec:greedy-partition} we formally define the partitioning scheme introduced in \cite{2013cass:gaussian-tail-rde}, and describe how the functional of interest behaves under path concatenation and strong convergence. In Section \ref{sec:lipschitz} we extend the improved Lipschitz estimate for RDEs driven by Gaussian noise \cite{2013friz:improved-lipschitz-estimate} to the regime-switching case. In Section \ref{sec:strong-convergence}, we establish the strong convergence of approximations to regime-switching RDEs by utilising tail behaviour of the greedy partition from Section \ref{sec:greedy-partition} and the pathwise estimates from Section \ref{sec:lipschitz}. We conclude with a brief application to the standard scheme of approximation via linear interpolation to a Brownian rough path.
 
\section{Preliminaries}\label{sec:preliminaries}

In this section we provide some background information on rough path theory, regime-switching SDEs and strong convergence. 
\subsection{Rough path theory}
Rough path theory, first introduced in \cite{lyons1998rough_paths}, provides a pathwise solution theory for stochastic differential equations. To make rigorous the notion of `roughness', we use the $p$-variation seminorm of a path $X:[0,T] \rightarrow \R^d$ as defined in \eqref{eq:p-var}. If one is interested in solving equations of the form
\begin{equation}\label{eq:rde}
    \di Y_t = Y_t \, \di X_t, \qquad Y_0 = y,
\end{equation}
for paths $X:[0,T] \rightarrow \R^d$ and $Y:[0,T] \rightarrow \mathcal{L}(\R^d, \R^e)$ (the space of linear maps from $\R^d$ into $\R^e$), $p$-variation turns out to be a natural scale with which to measure  the roughness of paths.  As Young \cite{young1936holder} discovered, the integral $\int_0^t Y_s \,\di X_s$ is defined classically via Riemann--Stieltjes integration if $Y$ and $X$ are of bounded $p$ and $q$-variation, with $\frac{1}{p} + \frac{1}{q} > 1$.  This condition is sharp in the sense that examples of paths exist satisfying $p^{-1} + q^{-1} =1$ such that the Riemann--Stieltjes approximations of $\int_0^t Y_s \, \di X_s$ fail to converge.  In the simplified setting where $Y_t = f(X_t)$ for some smooth $f:\R^d \rightarrow \mathcal{L}(\R^d, \R^e)$, this implies that $X$ must be a path with finite $p$-variation for $p < 2$. Remarkably, this is the \textit{exact} threshold which the sample path regularity of Brownian motion fails to meet \cite{friz2010introduction}. Thus, we call a path \textit{rough} if it is of finite $p$-variation only for some $p>2$. 

To deal with this problem when the driving path $X$ is a Brownian motion, stochastic calculus exploits the desirable probabilistic properties of Brownian motion by considering convergence of the approximates $\sum_{[u,v] \in \mathcal{P}} f(B_u) (B_v- B_u)$ in \textit{probability}, rather than almost surely, as the mesh size of the partition $\mathcal{P}$ tends to zero. Rough path theory, on the other hand, takes a different approach.

The perspective of rough path theory is that the limit 
\begin{equation*}
    \lim\limits_{|\mathcal{P}|\rightarrow 0} \sum_{[u,v] \in \mathcal{P}} f(X_u) (X_v - X_u)
\end{equation*}
fails to converge because the zeroth order approximation $f(X_r) \approx f(X_u)$ for $r \in [u,v]$ is not good enough to account for the roughness of the driving path $X$. If instead we take the first order approximation $f(X_r) \approx f(X_u) + Df(X_u)(X_r - X_u)$, the approximation of our integral becomes 
\begin{align}
   \int_0^t f(X_s) \, \di X_s &= \lim\limits_{|\mathcal{P}|\rightarrow 0} \sum_{[u,v] \in \mathcal{P}} \int_u^v \left( f(X_u) + Df(X_u) \big(X_r - X_u\big) \right) \di X_r \nonumber \\
    &= \lim\limits_{|\mathcal{P}|\rightarrow 0} \sum_{[u,v] \in \mathcal{P}} f(X_u) \big( X_v - X_u \big) + Df(X_u) \int_u^v \big(X_r - X_u\big) \otimes \di X_r, \label{eq:int-compensated}
\end{align}
where $\otimes$ denotes the tensor product. When $X$ has finite $p$-variation with $p<2$, so that $\int X \otimes \di X$ is classically defined, then the above approximation converges to the usual Riemann--Stieltjes integral. In the rough setting with $p \in [2,3)$, the iterated integral $\int X \otimes \di X$ is no longer uniquely defined as a function of the path $X$, and so we instead \textit{postulate} values for the iterated integral by specifying a two-parameter path $\mathbb{X}:[0,T]^2 \rightarrow \R^d \otimes \R^d$. Any choice of $\mathbb{X}$ satisfying an algebraic condition known as \textit{Chen's relation} and having finite $p/2$-variation is a suitable candidate, with different choices yielding different limits for \eqref{eq:int-compensated}. Equation \eqref{eq:int-compensated} is then referred to as the rough path integral with respect to the \textit{$p$-rough path} $\mathbf{X} = \big(X, \mathbb{X}\big)$.  This approach works not only for functions of the path $X$, but also for a class of paths known as \emph{controlled rough paths}.
\begin{definition}
        Let $X:[0,T] \rightarrow \R^d$ be a path of finite $p$-variation. We say that $Y:[0,T]\rightarrow \R^e$ is \emph{controlled by $X$} if there exists a path $Y' : [0,T] \rightarrow \mathcal{L}(\R^d, \R^e)$ such that 
        \begin{equation*}
            Y_{s,t} = Y_s'X_{s,t} + R_{s,t}^Y,
        \end{equation*}
        where $Y$ and $Y'$ have finite $p$-variation, and the remainder term $R^Y$ has finite $p/2$-variation. We denote the space of all paths controlled by $X$ by $\mathscr{D}^{p/2}_X$.
\end{definition}
The space $\mathscr{D}^{p/2}_X$ is a Banach space under the norm 
\begin{equation*}
    (Y, Y') \mapsto |Y_0| + |Y_0'| + \Vert Y, Y'\Vert_{X}^{p/2}
\end{equation*}
with 
\begin{equation*}
    \Vert Y, Y'\Vert_{X}^{p/2} := \Vert Y\Vert_{p-var;[0,T]} + \Vert R^Y\Vert_{p/2-var;[0,T]}.
\end{equation*}
This opens the door to fixed-point techniques, allowing us to establish the existence and uniqueness of solutions to \emph{rough} differential equations by finding fixed-points in $\mathscr{D}^{p/2}_X$.

Although this procedure of making a particular choice of $\mathbb{X}$ may seem arbitrary, there are some natural choices when $X=B$ is a Brownian motion. In this setting, we may simply enhance $B$ with its stochastic iterated integral $\mathbb{B}$, in either the It\^o or Stratonovich sense. Here, we work with the Stratonovich iterated integrals $\int B\otimes\,    \di B$, and refer to the object $\mathbf{B} := \big( B, \mathbb{B}\big)$ as the \textit{Brownian rough path}, or \textit{enhanced Brownian motion}.

Given the importance of the $p$-variation of a path highlighted above, we introduce the \emph{inhomogeneous $p$-rough path metric} 
\begin{equation*}
    \rho_{p-var}(\mathbf{X}, \mathbf{Y}) := \max \left\{ \Vert X - Y \Vert_{p-var;[0,T]}, \Vert\mathbb{X} - \mathbb{Y} \Vert_{p/2-var;[0,T]}\right\},
\end{equation*}
with 
\begin{equation*}
    \Vert\mathbf{X}\Vert_{p/2-var;[0,T]} = \left( \sup_{\mathcal{P}} \sum_{[u,v] \in \mathcal{P}} \vert \mathbb{X}_{u,v} \vert^{p/2} \right)^{2/p}.
\end{equation*}
Finally, we introduce the \emph{$\gamma$-Lipschitz} norm of a map between two normed spaces. In the following, let $\lfloor \gamma \rfloor$ be the largest integer which is \emph{strictly} smaller than $\gamma$, so that we may always write $\gamma = \lfloor \gamma \rfloor + \{\gamma\}$ with $\{\gamma\} \in (0, 1]$.
\begin{definition}[$\gamma$-Lipschitz norm, \cite{friz2010introduction}]
    A map $V:E\rightarrow F$ between two normed spaces $E$ and $F$ is called $\gamma$-Lipschitz (in the sense of E. Stein), in symbols 
        \begin{equation*}
            V \in \mathrm{Lip}^\gamma(E,F)\text{ or simply }V \in \mathrm{Lip}^\gamma \text{ if }E = F,
        \end{equation*}
    if $V$ is is $\lfloor \gamma \rfloor$ times continuously differentiable and such that there exists a constant $0 \leq M < \infty$ such that the supremum norm of its $k$th-derivatives $k=0,..., \lfloor \gamma \rfloor$, and the H\"older norm of its $\lfloor \gamma \rfloor$th derivative are bounded by $M$. The smallest $M$ satisfying the above conditions is the $\gamma$-Lipschitz norm of $V$ and denoted $|V|_{\mathrm{Lip}^\gamma}$.
\end{definition}

The payoffs to this approach are numerous. Whilst the solution map associated with the SDE $\di Y_t = f(B_t)\,  \di B_t$, $Y_0 = y_0$ is in general only measurable, the solution map associated to a \textit{rough} differential equation $\di Y_t = V(Y_t) \, \di \mathbf{X}_t$, $Y_0 = y_0$ is locally Lipschitz continuous (under some standard conditions) with respect to the initial condition, driving path, and the vector field $V$. This leads to error estimates of the form
\begin{equation}\label{eq:rde-sol}
    \Vert Y - \widetilde{Y}\Vert_{\infty} \leq C \big( |y_0 - \widetilde{y}_0| + |V - \widetilde{V}|_{Lip^\gamma} + \rho_{p-var}(\mathbf{X}, \widetilde{\mathbf{X}}) \big).
\end{equation}

In particular, this estimate may be applied when we consider approximations of  rough differential equations (RDEs) of the form \eqref{eq:rde}.  Given a rough path $\mathbf{X} = (X, \mathbb{X})$ and a collection of finite-variation approximations $\{X^\lambda\}_{\lambda \geq 0}$ of $X$, we can \textit{lift} these approximations to rough paths by enhancing them with the Riemann--Stieltjes integrals $\mathbb{X}^\lambda_{s,t} := \int_s^t X_{s,r}^\lambda \otimes \di X^\lambda_r$. Establishing convergence of $\mathbf{X}^\lambda := (X^\lambda, \mathbb{X}^\lambda)$ to $\mathbf{X}$ in rough path metric then implies convergence of $Y^\lambda \rightarrow Y$ via Equation \eqref{eq:rde-sol}, where $Y^\lambda$ is the solution of $\di Y_t = f(Y_t) \di \mathbf{X}_t$ and $Y^\lambda_t$ is the solution of $\di Y_t^\lambda = f(Y^\lambda_t) \di \mathbf{X}^\lambda_t$. 
\begin{remark}
     The supremum norm in Equation \eqref{eq:rde-sol} can generally be replaced by the $p$-variation rough path metric. However, we return to the weaker supremum topology as it is a standard way to measure approximation error.
\end{remark}
\subsection{Gaussian rough paths}

Unlike the semimartingale setting, there is no notion of stochastic integration we could use to lift a general $\R^d$-valued Gaussian process $X$ to a rough path $\mathbf{X}$. Recalling that the law of a (centred) Gaussian process $X_t = (X^1_t, ..., X^d_t)$ is completely determined by its covariance function $R(s,t) = \E[X_s \otimes X_t]$, it is perhaps not surprising that the question of finding a natural rough path lift boils down to the regularity of the \textit{rectangular increments} of $\R^d$.
\begin{definition}
    Let $X$ be a centred $\R^d$-valued Gaussian process with covariance function $R:[0,T]^2 \rightarrow \R^{d \times d}$. The rectangular increments of $R$ are given by
    \begin{equation*}
        R\begin{pmatrix}
            s&,& t \\
            s'&,& t'
        \end{pmatrix}:= \E[ X_{s,t} \otimes X_{s', t'}].
    \end{equation*}
    Given a set $I\times I' \subset [0,T]^2$, the 2D $\varrho$-variation of $R$ is given by 
    \begin{equation*}
        \Vert R\Vert_{\varrho, I \times I'} = \left( \sup\limits_{\substack{\mathcal{P}\subset I, \\ \mathcal{P}' \subset I'}}  \sum_{\substack{[s,t] \in \mathcal{P}, \\ [s', t'] \in \mathcal{P'}}} \left|  R\begin{pmatrix}
            s&,& t \\
            s'&,& t'
        \end{pmatrix} \right|^{\varrho}\right)^{1/\varrho}
    \end{equation*}
\end{definition}
Under the condition that $X$ is a Gaussian process with covariance $R$ of finite 2D $\varrho$-variation for $\varrho <2$, we may define the integral of $X^i$ against $X^j$ as the $L^2$-limit 
\begin{equation}
    \int_s^t X^i_{s,r} \,\di X^j_r = \lim\limits_{|\mathcal{P}|\rightarrow 0} \sum_{[u,v]\in \mathcal{P}} X^i_{s,v} X^j_{u,v}.
\end{equation}
Further to this, we have an estimate of the form
\begin{equation*}
    \E \left[ \left( \int_s^t X^i_{s,r} \, \di X^j_r \right)^2\right] \leq C\Vert R\Vert_{\varrho; [s,t]^2},
\end{equation*}
where $C = C(\varrho)$ [Proposition 10.3, \cite{friz2014rough-path-book}]. Defining 
\begin{equation*}
    \mathbb{X}^{i,j}_{s,t} = \int_s^t X^i_{s,r}\ \di X_r^j, \quad \mathbb{X}^{i,i}_{s,t} = \frac{1}{2}\left(X^i_{s,t}\right)^2,\quad \text{and} \quad \mathbb{X}_{s,t}^{i,j} = - \mathbb{X}_{s,t}^{i,j} +X^i_{s,t} X^j_{s,t},
\end{equation*}
it follows that if there exists $M$ and $\varrho \in \big[1, \frac{3}{2}\big)$ such that for every $i \in \{1,..., d\}$ and $0\leq s \leq t \leq T$
\begin{equation*}
    \Vert R_{X^i}\Vert_{\varrho; [s,t]^2} \leq M|t-s|^{1/\varrho},
\end{equation*}
then the process $(X, \mathbb{X})$ is almost surely a $p$-rough path for $p> 2\varrho $ [Theorem 10.4, \cite{friz2014rough-path-book}]. This result can be pushed further to $\varrho \in \big[1, 2\big)$, but requires a \textit{third} level of the rough path to be defined (see Chapter 15 of \cite{friz2010introduction}).
\begin{example}[Brownian motion]
Given a Brownian motion $B = (B^1, ..., B^d)$, the covariance function has the form $R(s,t) = (s\wedge t)I_d$. Since $B^i$ and $B^j$ are centred and independent for $i \neq j$, $\E[X^i_{s,t} X^j_{s,t}]=0$. The on-diagonal terms then take the form
\begin{align*}
    \E[X^{i}_{s,t} X^i_{s', t'}] &= \E[ X^i_{t'} X^i_t] - \E[X_{s'}^iX^i_t] + \E[X_s^i X_{s'}^i] - \E[X_s^i X_{t'}^i]\\
    &= \min\{t, t'\} - \min\{s', t\} + \min\{s, s'\} - \min\{s, t'\}\\
    &=|[s', t'] \cap [s, t]|,
\end{align*}
showing that $B$ has finite 2D $\varrho$-variation for $\varrho=1$. Lifting $B$ to a rough path $\mathbf{B}$ using the Gaussian framework will yield a process $\mathbf{B}$ which is indistinguishable from the Stratonovich lift $(B, \mathbb{B}^{\text{Strat}})$ \cite{friz2014rough-path-book}.
\end{example}
\begin{example}[Fractional Brownian motion]
    Recalling fractional Brownian motion $B^H$ for $H\in \big(0, \frac{1}{2}\big]$ is the $\R^d$-valued process with covariance function 
    \begin{equation*}
        R^H(s,t) = \frac{1}{2} \left(t^{2H} + s^{2H} - |t-s|^{2H} \right),
    \end{equation*}
\cite{friz2010introduction} show that $B^H$ has finite 2D $\varrho$-variation for $\varrho = \frac{1}{2H}$.
\end{example}

\subsection{Regime-switching SDE via rough paths}

The consistency between stochastic and rough integration, when both are defined, is a well-known feature of rough path theory \cite{friz2010introduction}. Indeed, the standard regularity requirements that are imposed to guarantee the existence of a unique solution to a given RDE are enough to ensure the existence of a unique strong solution to the corresponding SDE. 
\begin{definition}[Regime-switching SDE]\label{def:rssde}
    Let $J$ be a jump process with finite activity on compact intervals and countable state space $\mathcal{E}$. Let $\mu: \R_+\times \R^m \times \mathcal{E} \rightarrow \R^m$ and $\sigma : \R_+ \times \R^m \times \mathcal{E} \rightarrow \mathcal{L}(\R^d, \R^m)$, and write $\mu(t,x,i) =\mu_i(t,x)$, $\sigma(t,x,i) = \sigma_i(t,x)$. Finally, let $B = (B^1,..., B^d)$ be an $\R^d$-valued Brownian motion independent of $J$ and $Y_0 \in \R^m$. Then the equation
    \begin{equation}\label{eq:rssde-def}
        Y_t = Y_0 + \int_0^t \mu_{J_s}(s,Y_s) \, \di s + \int_0^t \sigma_{J_s}(s,Y_s)\, \di B_s
    \end{equation}
    is referred to as a \textit{regime-switching} stochastic differential equation.
\end{definition}
\begin{remark}
    One can express the solution as a coupled process $\big\{(Y_t, J_t)\big\}_{t\geq 0}$ where $J$ is given by the stochastic integral with respect to a Poisson random measure $\mathfrak{p}$ \cite{yin2010hybrid_switching}. This can be useful in refining some arguments when working with jump processes $J$ with probabilistic behaviour depending on the trajectory of $Y$, but for our purposes \eqref{eq:rssde-def} will be more useful. 
\end{remark}
\begin{remark}\label{rem:rssde-concat}
    Under the conditions on $J$ in Definition \ref{def:rssde}, with suitable conditions on $\mu$ and $\sigma$, there exists a unique strong solution for the SDEs 
    \begin{equation*}
        Y_t^i = Y_0^i + \int_0^t \mu_i(s,Y_s^i)\, \di s + \int_0^t \sigma_{i} (s, Y_s^i) \, \di B_s\,.
    \end{equation*}
    Considering the stopping times $\tau_0 = 0$ and $\tau_{k+1} = \inf\{t> \tau_k : J_t \neq J_{\tau_k} \}$, we may write for $t \in [ \tau_k, \tau_{k+1}]$
    \begin{align*}
        Y_t &= Y_0 + \int_0^t \mu_{J_s}(s, Y_s)\, \di s + \int_0^t \sigma_{J_s}(s, Y_s)\, \di B_s \\
        &=  Y_0 + \int_0^{\tau_k} \mu_{J_s}(s, Y_s)\, \di s +\int_{\tau_k}^t \mu_{J_s}(s, Y_s)\, \di s + \int_0^{\tau_k} \sigma_{J_s}(s, Y_s)\, \di B_s  + \int_{\tau_k}^t \sigma_{J_s}(s, Y_s)\, \di B_s \\
        &= Y_{\tau_k} + \int_{\tau_k}^t \mu_{J_{\tau_k}}(s, Y_s)\, \di s+ \int_{\tau_k}^t \sigma_{J_{\tau_k}}(s, Y_s)\, \di B_s\,.
    \end{align*}
    The almost sure finite jump activity of $J$ then implies that $\lim_{k\rightarrow \infty} \tau_k = \infty$, so that the above holds for all $t\in [0,T]$, $T>0$. Thus, we can recover solutions of \eqref{eq:rssde-def} by concatenating solutions of SDEs within known regimes. 
\end{remark}
We now introduce the rough equivalent of regime-switching SDEs: 
\begin{definition}\label{def:rp-rssde}
    Let $J$ be a jump process with a.s. finite activity on compact intervals and countable state space $\mathcal{E}$. Let $p, \gamma$ be such that $2<p<\gamma$. Assume that 
    \begin{enumerate}
        \item $V^j = (V_{i}^j)_{1\leq i \leq d}$ is a collection of $\mathrm{Lip}^\gamma(\R^e)$ vector fields for $ j \in \mathcal{E}$;
        \item $\sup_{j \in \mathcal{E}} |V^j|_{\mathrm{Lip}^\gamma} < \infty$; and
        \item $\mathbf{X}$ is a Gaussian rough path.
    \end{enumerate}
    Let $(\tau_k)_{k=0}^{N^J+1}$ be the jump times of $J$ on $[0,T]$ (with $N^J$ denoting the number of jumps of $J$) and let $\{Y^k\}_{k =0}^{N^J}$ be the RDE solutions to 
    \begin{align*}
        \di Y_t^0 &= y_0 + \int_0^t V_{J_0}(Y_s^0) \, \di \mathbf{X}_s,  & t &\in [0, \tau_1], \quad y_0 \in \R^e, \\
        \di Y_t^k &= Y_{\tau_k}^{k-1} + \int_{\tau_k}^t V_{J_{\tau_k}}(Y_s^k) \, \di \mathbf{X}_s\,, & t &\in [\tau_k, \tau_{k+1}]\,, \quad 1\leq k \leq N.
    \end{align*}
Then the path $\{Y_t\}_{t\geq 0}$ constructed by concatenating the individual paths $\{Y^k\}_{k=0}^{N^J}$ is said to solve the regime-switching rough differential equation driven by $\mathbf{X}$ and $J$. 
\end{definition}

It remains to prove consistency between stochastic and rough regime-switching differential equations. Given that solutions to these equations are determined via fixed-points of the rough and stochastic integral respectively, we prove consistency for general regime-switching rough integration and regime-switching stochastic integration.

\begin{theorem}
    Let $\mathbf{B}=\big(B, \mathbb{B}\big)$ denote the It\^o Brownian rough path, $J$ be a Markov chain with finite state space $\mathcal{E}$ with $\E[N^J]<\infty$, and suppose $\big( X^i, (X^i)'\big) \in \mathscr{D}_{B(\omega)}^{p/2}$ almost surely for each $i \in \mathcal{E}$. Then the regime-switching rough integral 
    \begin{equation}\label{eq:rough-integral-approx}
        \int_0^t X^{J_s} \, \di \mathbf{B}_s  := \lim_{n\rightarrow \infty} \sum_{[u,v] \in \mathcal{P}_n}\left( X_u^{J_u} B_{u,v} + \big(X^{J_u}_{u}\big)' \mathbb{B}_{u,v} \right),
    \end{equation}
    exists, with the limit taken along any sequence $\{\mathcal{P}_n\}_{n=1}^\infty$ with mesh size tending to 0. In the case that $X^i$ and $(X^i)'$ are adapted for every $i \in \mathcal{E}$, then 
    \begin{equation*}
        \int_0^t X^{J_s}\, \di \mathbf{B}_s = \int_0^t X^{J_s} \, \di B_s,
    \end{equation*}
    almost surely.
\end{theorem}
\begin{proof}
    In the following we assume that 
    \begin{equation*}
        \E \left[ \sup_{i \in \mathcal{E}} \Vert X^i, \big(X^i\big)' \Vert_{B}^{p/2} \right] < \infty,    
    \end{equation*}
    noting that if this is not the case we may proceed by localisation. Having defined the regime-switching rough integral for a fixed trajectory of $J(\omega)$ and $\mathbf{B}(\omega)$ via concatenation of the rough integral over periods when $J$ is in a fixed regime, we must show that the approximation in Equation \eqref{eq:rough-integral-approx} converges to the same path. To this end, let $I^c_t$ denote the integral achieved by concatenation and $I_t$ by the proposed limit \eqref{eq:rough-integral-approx}. Let $(\tau_k)_{k=0}^{N^J+1}$ denote the jump times of $J$ on the interval $[0,T]$. Noting that $I_t = I_t^c$ for  $t \in [0, \tau_1]$, we consider $t \in [\tau_1, \tau_2)$. We must deal with the fact that for any sequence of partitions $\{\mathcal{P}_n\}_{n=1}^\infty$, there will generally be consecutive points $u^n, v^n \in \mathcal{P}_n$ with $u^n < \tau_1 < v^n$ such that we will be approximating the regime-switching rough integral by the wrong control. That is, if $J_{u^n} = i$ and $J_{v^n} = j$, we will be approximating $I_t$ over the interval $[u^n,v^n)$ as if the jump process was in state $i$ over the whole interval. Adding and subtracting the `correct' control we see that 
    \begin{align*}
        \bigg\vert \sum_{[u,v] \in \mathcal{P}_n} &\left( X_{u}^{J_u} B_{u,v} + \big(X^{J_{u}}_{u}\big)' \mathbb{B}_{u,v}\right)  \pm \left( X_{\tau_1}^{J_{\tau_1}} B_{\tau_1, v^n} + \big( X_{\tau_1}^{J_{\tau_1}}\big)' \mathbb{B}_{\tau_1, v^n} \right) \bigg\vert \\
        &\leq \bigg\vert \sum_{[u,v] \in \mathcal{P}_n'} X^{J_u}_u B_{u,v}  + \big(X^{J_u}_{u}\big)' \mathbb{B}_{u,v} \bigg\vert  + \left\vert  \big(X_{u^n}^i - X_{v^n}^j \big) B_{\tau_1, v^n} \right\vert \\
        & \qquad \qquad + \left\vert \left( \big( X^i_{u^n} \big)' - \big(X^j_{\tau_1}\big)'\right)\mathbb{B}_{\tau_1, v^n} \right\vert + \left\vert  \big( X^i_{u^n}\big)' B_{u^n, \tau_1} \otimes B_{\tau_1, v^n} \right\vert,
    \end{align*}
   where $\mathcal{P}_n' = \mathcal{P}_n \cup \{\tau_1\}$ and we have used the fact that $B$ is additive and $\mathbb{B}$ satisfies Chen's relation to split the terms. Since $\tau_1 \in \mathcal{P}_n'$ for every $n$, the first term converges to $I^c_t$. The other three terms do not affect the limit, which we see from the bounds 
   \begin{align*}
    \left\vert  \big(X_{u^n}^i - X_{v^n}^j \big) B_{\tau_1, v^n} \right\vert &\leq C \Vert B\Vert_{\alpha} \vert v^n - \tau_1 \vert^\alpha = O\big(\vert \mathcal{P}_n'\vert^\alpha \big) \\
    \left\vert \left( \big( X^i_{u^n} \big)' - \big(X^j_{\tau_1}\big)'\right)\mathbb{B}_{\tau_1, v^n} \right\vert &\leq C \Vert \mathbb{B}\Vert_{2\alpha} |v^n - \tau_1|^{2\alpha} = O\big(\vert \mathcal{P}_n'\vert^\alpha \big) \\
    \left\vert  \big( X^i_{u^n}\big)' B_{u^n, \tau_1} \otimes B_{\tau_1, v^n} \right\vert &\leq C \Vert B \Vert_{\alpha}^2 \vert v^n - \tau_1\vert^\alpha \vert \tau_1 - u^n \vert^\alpha = O\big(\vert \mathcal{P}_n'\vert^\alpha \big),
   \end{align*}
   for 
   \begin{equation*}
    C:= 2\max\left\{ \sup_{i \in \mathcal{E}} \Vert X^i\Vert_{\infty},\, \sup_{i \in \mathcal{E}} \Vert (X^i)'\Vert_\infty \right\} < \infty.
\end{equation*}
We repeat this process for $[\tau_2, \tau_3)$, $[\tau_3, \tau_4)$ and so on, noting that since $N^J< \infty$ almost surely we can cover $[0,T]$ in a finite number of steps, showing that $I_t = I_t^c$ almost surely for any $t \in [0,T]$. 

Next, we show that the rough path and stochastic integrals coincide almost surely. To do so, we use the standard technique of showing that our rough path approximation converges in $L^2$ to the It\^o integral $\int_0^t X^{J_s} \di B_s$. Then, since the approximation converges to the rough path integral almost surely and the stochastic integral in $L^2$, these limits must coincide. We consider the specific partition $\mathcal{P}_n = \{tk/n\}_{k=0}^n$ of $[0,t]$. Consider the $L^2$ error 
    \begin{align*}
        \E[I_t^n] &=\E \bigg[ \bigg\vert \int_0^t X^{J_s} \di B_s - \sum_{k=0}^{n-1} X_{t_k^n}^{J_{t_k}} B_{t_k^n, t_{k+1}^n} + ( X_{t_k^n}^{J_{t_k}})' \mathbb{B}_{t_k^n, t_{k+1}^n} \bigg\vert^2\bigg] \\
            &= \E \left[ \left\vert \sum_{k=0}^{n-1} \int_{t_k^n}^{t_{k+1}^n} \left( X^{J_s}_s - X_{t_k^n}^{J_{t_k}} - (X_{t_k^n}^{J_{t_k}})' B_{t_k^n, s} \right) \di B_s \right\vert^2 \right]\\
            &= \E \left[ \E \left[ \left\vert \sum_{k=0}^{n-1} \int_{t_k}^{t_{k+1}} \left( X^{J_s}_s - X_{t_k^n}^{J_{t_k}} - (X_{t_k^n}^{J_{t_k}})' B_{t_k^n, s} \right) \di B_s \right\vert^2 \big\vert \mathcal{F}_t^J \right] \right],
    \end{align*}
where in the final step we condition on the $\sigma$-algebra $\mathcal{F}^J_t := \sigma(J_s, 0 \leq s \leq t)$ generated by $\{J_t\}_{t \in [0,T]}$. Noting that the product of the stochastic integral over disjoint intervals has expectation zero and applying the (multivariate) It\^o isometry to the remaining terms, we get
\begin{align*}
    \E[I_t^n] &= \E \left[ \E\left[ \sum_{k=0}^{n-1} \int_{t_k^n}^{t_{k+1}^n}  \left| X_s^{J_s} - X_{t_k^n}^{J_{t_k^n}} - (X_{t_k^n}^{J_{t_k^n}})' B_{t_k^n, s}  \right|^2 \di s\big\vert \mathcal{F}_t^J\right] \right].
\end{align*}
Now, if $J_s = i$ for all $s \in [t_k^n, t_{k+1}^n]$, then the integrand is controlled by $B$ for the fixed regime over the whole interval, meaning that $X^i_{t_k^n, s} = (X^i_{t_k^n})' B_{t_k^n, s} + R_{t_k^n, s}^i$, where $R^i \in C^{2\alpha}$ is the remainder term. In such a case, we can bound the integral by writing 
\begin{align*}
    \int_{t_k^n}^{t_{k+1}^n}  \left| X_s^{i} - X_{t_k^n}^{i} - (X_{t_k^n}^i)' B_{t_k^n, s}  \right|^2 \di s &= \int_{t_k^n}^{t_{k+1}^n} | R_{t_k^n, s}^i |^2 \di s \\
    &\leq \Vert R^i \Vert_{2\alpha} \int_{t_k^n}^{t_{k+1}^n} |s- t_{k^n}|^{4\alpha} \di s \\
    &\leq \sup_{i \in \mathcal{E}} \Vert R^i\Vert_{2\alpha} \cdot \frac{1}{4\alpha + 1} |t^n_{k+1} - t_k^n|^{1 + 4\alpha}\\
    &\leq \sup_{i \in \mathcal{E}} \Vert R^i\Vert_{2\alpha} \cdot \frac{1}{4\alpha + 1} \left(\frac{t}{n} \right)^{1 + 4\alpha}
\end{align*}
In the case that a jump occurs at time $\tau$ to state $j$ in the interval $[t_k^n, t_{k+1}^n]$, we decompose the integral via 
\begin{align*}
    \int_{t_k^n}^{t_{k+1}^n}  \left| X_s^{J_s} - X_{t_k^n}^{i} - (X_{t_k^n}^i)' B_{t_k^n, s}  \right|^2 \di s &=  \int_{t_k^n}^{\tau}  \left| X_s^{i} - X_{t_k^n}^{i} - (X_{t_k^n}^i)' B_{t_k^n, s}  \right|^2 \di s  \\ 
    & \qquad +  \int_{\tau}^{t_{k+1}^n}  \left| X_s^{j} - X_{t_k^n}^{i} - (X_{t_k^n}^i)' B_{t_k^n, s}  \right|^2 \di s.
\end{align*}
The first term can be bounded as before, while for the second we observe that 
\begin{align*}
    \int_{\tau}^{t_{k+1}^n}  \left\vert X_s^{j} - X_{t_k^n}^{i} - (X_{t_k^n}^i)' B_{t_k^n, s}  \right\vert^2 \di s &= \int_{\tau}^{t_{k+1}^n}  \left| (X_s^j - X_s^i) + R_{t_k^n, s}^i \right|^2 \di s \\
    &\leq 2 \sup_{i \in \mathcal{E}} \Vert X^i\Vert_\infty^2 |t_{k+1}^n - \tau| \\ 
    & \qquad \quad + \sup_{i \in \mathcal{E}} \Vert R^i\Vert_{p/2} \frac{|t_{k+1}^n - \tau|^{4\alpha + 1}}{4\alpha + 1}\\
    &\leq C\left( \frac{t}{n} +\frac{1}{4\alpha+1} \left(\frac{t}{n}\right)^{4\alpha + 1} \right).
\end{align*}
Since this occurs in at most $N^J$ of the $n$ intervals, we can bound $\E[I_n]$ (absorbing constants dependent on $\alpha$ and $t$) by 
\begin{align*}
    \E[I_n] &\leq C \E \left[ \sup_{i\in\mathcal{E}} \Vert R^i \Vert_{2/p} \left( n^{-4\alpha} + \frac{N^J}{n} + \frac{N^J}{n^{4\alpha + 1}}    \right) \right] \\
    &\leq C \E\left[ \sup_{i \in \mathcal{E}}\Vert X^i, (X^i)'\Vert_{B}^{p/2} \right]\left( n^{-4\alpha} + \E[N^J]\big(n^{-1} + n^{-(1+4\alpha)}\big)  \right) \rightarrow 0,
\end{align*}
provided $\E[N^J] < \infty$ and $\E\left[ \sup_{i \in \mathcal{E}} \Vert X^i, (X^i)'\Vert_{B}^{p/2}\right] < \infty$.
\end{proof}

\subsection{Strong convergence}
Let $(M, d)$ be a metric space, and suppose that $\{X^\lambda\}_{\lambda \geq 0}$ is a family of $M$-valued random variables with time horizon $[0,T]$. We say that $X^\lambda$ converges strongly to an $M$-valued random variable $X$ if 
\begin{equation*}
     \Pb \left( \lim\limits_{\lambda \rightarrow \infty} d(X^\lambda, X) = 0 \right) = 1\,. 
\end{equation*}
Further, we say that $X^\lambda$ converges \textit{with rate} $\delta(\lambda)$ if there exists a function $\delta: \mathbb{R}^+\rightarrow \mathbb{R}^+$ and constant $\alpha(q,T)$ such that 
\begin{equation*}
    \Pb \left( d(X^\lambda, X) \geq \alpha \delta(\lambda) \right) = o(\lambda^{-q})\,,
\end{equation*}
for all $q >0$. Of particular interest is the case when we take $M$ to be the path space of a given stochastic process. For example, we may take the usual choice $M = C\big( [0,T], \mathbb{R} \big)$ with $d(X,Y) = \sup_{t \in [0,T]} |X_t - Y_t |$ to investigate the strong convergence of stochastic processes with continuous sample paths, as in \cite{romisch1985,nguyen2021rate}. In the sequel, we consider strong convergence in the rough path space $C^{p-\text{var}}\big([0,T], \mathbb{R}^d\big)$ equipped with the $p$-variation rough path metric. 

There are two approaches we will take to prove strong convergence. The first is a rudimentary application of the Markov inequality, while the second utilises Lipschitz estimates on bounded sets in $M$. The second method is well suited to extend strong convergence of processes that drive regime-switching SDEs to the solutions of said RSSDEs. 

\section{The greedy partition}\label{sec:greedy-partition}
We introduce the greedy partition first introduced in \cite{2013cass:gaussian-tail-rde} by Cass, Litterer and Lyons. In the following, let $\Delta_{[s,t]}$ be shorthand for the 2-simplex on $[s,t]$:
\begin{equation*}
    \Delta_{[s,t]} := \big\{ (u, v): s\leq u \leq v \leq t \big\}.
\end{equation*}
 We first recall the definition of a \textit{control}.
\begin{definition}
    A function $\omega: \Delta_{[0,T]} \rightarrow [0, \infty)$ is called a \textit{control} if:
    \begin{enumerate}
        \item $\omega$ is superadditive, so that for all $s \leq u \leq t \in [0,T]$ 
        \begin{equation*}
            \omega(s,u) + \omega(u,t) \leq \omega(s,t),
        \end{equation*}
        \item $\omega(s,s) = 0$ for all $s \in [0,T]$, and 
        \item $\omega$ is continuous. 
    \end{enumerate}
\end{definition}

\begin{definition}\label{def:greedy-partition}
    Let $\omega: \Delta_{[s,t]} \rightarrow [0, \infty)$ be a control. For $\alpha >0$, set
    \begin{align*}
        \tau_0(\alpha) &= s\,, \\
        \tau_{i+1}(\alpha) &= \inf \{ u: \omega(\tau_i, u) \geq \alpha, \tau_i(\alpha) < u \leq t \} \wedge t \,, \qquad i \geq 0
    \end{align*}
    and define 
    \begin{equation*}
        N_{\alpha, [s,t]}(\omega) = \sup \{ n \in \mathbb{N} \cup \{0\} : \tau_n(\alpha) < t \} \,.
    \end{equation*}
    The sequence $(\tau_i(\alpha))_{i=0}^\infty$ is then called the \textit{greedy sequence}, with $N_{\alpha,[s,t]}(\omega)+1$ counting the number of distinct elements in $(\tau_i(\alpha))_{i=0}^\infty$.
\end{definition}
    Lemma 4.9 and Corollary 4.10 in \cite{2013cass:gaussian-tail-rde} establish that $N_{\alpha, [s,t]}(\omega)$ is well-defined for the particular choice $\omega (s,t) = \Vert \mathbf{X}\Vert_{p-var;[s,t]}^p$ whenever $\mathbf{X}$ is a $p$-rough path, and that the greedy sequence (with the trivial tail removed) forms a partition of $[s,t]$. As discussed in \cite{friz2012integrability}, the random functional $N_{\alpha, [s,t]}(\Vert \mathbf{X}\Vert_{p-var;[s,t]}^p)$ enjoys far better  probabilistic  tail estimates than $\Vert \mathbf{X}\Vert_{p-var;[s,t]}^p$. This, combined with the Lipschitz estimate of Theorem 4 in \cite{2013friz:improved-lipschitz-estimate}, is the primary tool we will utilise in the proof of our main theorem. 
    \begin{figure}
        \includegraphics[width=0.8\linewidth]{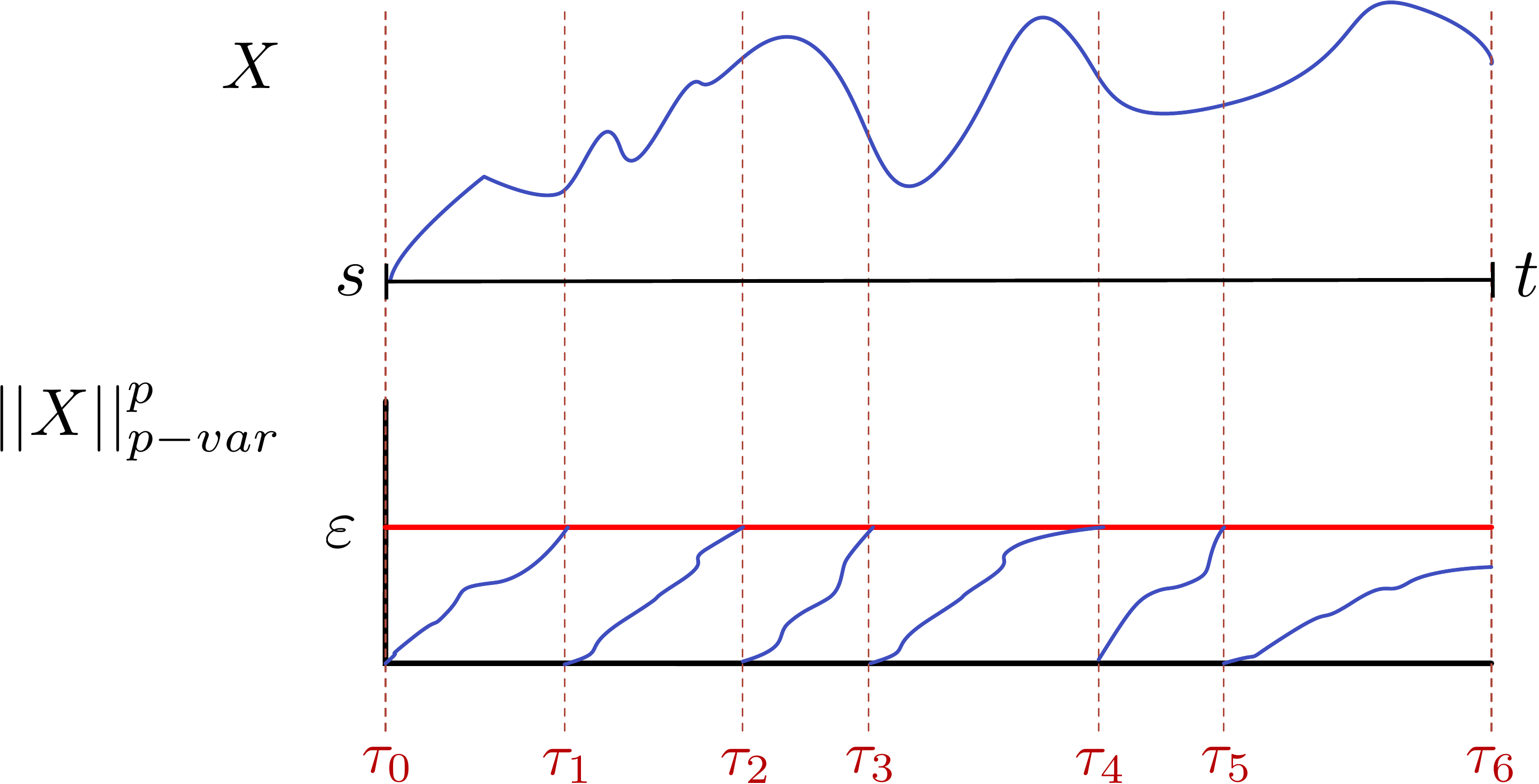}
        \caption{The greedy partition for a path $X$ with control $\Vert X\Vert_{p-var}^p$ and error $\alpha = \varepsilon$.}
    \end{figure}
    Specifically, we use the tail estimate from Corollary 2 of \cite{friz2012integrability},
    \begin{equation*}
        \Pb \left( N_{\alpha, [0,T]} \big( \Vert \mathbf{X}\Vert_{p-var;[0,T]}^p \big) > u \right) \leq \exp \left\{ - \frac{1}{2} \left( \hat{a} + \frac{\alpha^{1/p} u^{1/q}}{CK} \right)^2 \right\},
    \end{equation*}
    which applies to a class of Gaussian $p$-rough paths $\mathbf{X}$ including Brownian motion, and where $C>0, K>0, q>0, \hat{\alpha} > -\infty$ are constants dependent on the type of process $\mathbf{X}$ chosen.  Since for every polynomial $\sum_{i=1}^n a_i x^i$ there exists some $x_0, C_1, C_2>0$ such that 
    \begin{equation*}
        C_1 x^n \leq \sum_{i=1}^n a_i x^i \leq C_2 x^n,  \text{ for } x > x_0,
    \end{equation*}
    there exists some constant $c$ such that 
    \begin{equation*}
        \Pb \left( N_{\alpha, [0,T]} \big( \Vert \mathbf{X}\Vert_{p-var;[0,T]}^p\big) > u \right) \leq \exp \left\{ - c u^{2/q} \right\}
    \end{equation*}
    for all $u > u'$, where $u'$ depends on $c$ and $c = c(\hat{a}, C, K, p, q)$ depends on the previous constants. 
    \begin{remark}
        For the Brownian rough path case $\mathbf{X} = \mathbf{B}$, we may choose $q=1$, yielding Gaussian tails. 
    \end{remark}

    We now investigate the behaviour of $N_{\alpha, [\cdot, \cdot]}$ under concatenation.

\begin{lemma}\label{lem:subadditivity-greedy}
    Let $\{s=p_0< p_1< \cdots < p_{M+1} = t\}$ be a partition of $[s,t]$. Then 
    \begin{equation*}
        N_{\alpha,[s,t]}(\omega) \geq \sum_{k=0}^M N_{\alpha,[p_k, p_{k+1}]}(\omega)\,.
    \end{equation*}
\end{lemma}
\begin{proof}
    Let $(\tau_i(\alpha))$ be the greedy sequence over $[s,t]$ for the control $\omega$. If there exist $i,k$ such that $\tau_i < p_k < p_{k+1} < \tau_{i+1}$, then $N_{\alpha, [p_k, p_{k+1}]}(\omega)=0$, which follows from 
    \begin{equation*}
        N_{\alpha, [p_k, p_{k+1}]}(\omega) \leq N_{\alpha, [\tau_i, \tau_{i+1}]}(\omega)= 0\,.    
    \end{equation*}
    Thus,  without loss of generality we assume that  there is at most one $p_k$ between any given pair $\tau_i < \tau_{i+1}$. Next, we set $k_i := \sup\{k: \tau_k \leq p_i \}$ for $i \geq 0$ and write 
    \begin{align*}
        \alpha N_{\alpha, [s,t]} &= \sum_{k=0}^{N_{\alpha,[s,t]}(\omega)-1} \omega(\tau_k, \tau_{k+1}) \\
            &= \sum_{i=1}^{M+1} \sum_{j = k_{i-1}}^{k_i-1} \omega(\tau_j, \tau_{j+1}) \\
            &= \alpha \sum_{i=1}^{M+1} N_{\alpha, [\tau_{k_{i-1}}, \tau_{k_{i}+1}]}(\omega) \\
            &\geq \alpha \sum_{i=1}^{M+1} N_{\alpha, [p_{i-1}, p_i]}(\omega)
    \end{align*}
    with the last inequality following from the fact that $[p_{i-1}, p_i] \subseteq [\tau_{k_{i-1}}, \tau_{k_i+1}]$.

\end{proof}

We now investigate the tail behaviour of an approximation $\{\mathbf{X}^\lambda\}_{\lambda>0}$ to $\mathbf{X}$, under the assumption of strong convergence. We begin with a simple set inclusion:

\begin{lemma}\label{lem:tail-inclusion}
    Let $\mathbf{X}$ and $\mathbf{X}^\lambda$ be $p$-rough paths, and take $\alpha>0$. Then
    \begin{equation*}
        \left\{ N_{\alpha,[s,t]} \big(\Vert \mathbf{X}^\lambda\Vert_{p-var}^p \big) \geq u \right\} \subset \left\{ N_{\frac{\alpha}{2^{p-1}},[s,t]}\left(\Vert \mathbf{X}\Vert_{p-var}^p + \Vert \mathbf{X} -\mathbf{X}^\lambda\Vert_{p-var}^p \right) \geq u \right\}.
    \end{equation*}
\end{lemma}
\begin{proof}
    By the triangle inequality and Jensen's inequality, we have 
    \begin{equation*}
        \Vert \mathbf{X}^\lambda\Vert_{p-var}^p \leq 2^{p-1} \left( \Vert \mathbf{X}\Vert_{p-var}^p + \Vert \mathbf{X} - \mathbf{X}^\lambda\Vert_{p-var}^p  \right)\,.
    \end{equation*}
    Then, if $\Vert \mathbf{X}^\lambda\Vert_{p-var}^p \geq \alpha$ it follows that 
    \begin{equation*}
        \Vert \mathbf{X}\Vert_{p-var}^p + \Vert \mathbf{X} - \mathbf{X}^\lambda\Vert_{p-var}^p \geq \frac{\Vert \mathbf{X}^\lambda\Vert_{p-var}^p}{2^{p-1}} \geq \frac{\alpha}{2^{p-1}}\,,
    \end{equation*}
    yielding
    \begin{equation*}
        \left\{ \Vert \mathbf{X}^\lambda\Vert_{p-var; [s,t]}^p \geq \alpha \right\} \subset \left\{ \Vert \mathbf{X}\Vert_{p-var; [s,t]}^p + \Vert \mathbf{X} - \mathbf{X}^\lambda\Vert_{p-var; [s,t]}^p \geq \frac{\alpha}{2^{p-1}} \right\}\,.
    \end{equation*}
    For simplicity of notation, let us write 
    \begin{align*}
        \omega^1 (s,t) &= \Vert \mathbf{X}^\lambda\Vert ^p_{p-var;[s,t]}, \\
        \omega^2 (s,t) &= \Vert \mathbf{X}\Vert_{p-var;[s,t]}^p + \Vert \mathbf{X} - \mathbf{X}^\lambda \Vert_{p-var;[s,t]}^p \,.
    \end{align*}
    Consider the greedy sequence $\tau_i^1(\alpha)$ (resp. $\tau_j^2(\alpha/2^{p-1})$) associated with the control $\omega^1$ (resp. $\omega^2$). Since $\omega^1(\tau_i^1, \tau_{i+1}^1) = \alpha$ for all $i =0,...,N_{\alpha,[s,t]}(\omega^1)-1$, we have $\omega^2(\tau_i^1, \tau_{i+1}^1) \geq \alpha/2^{p-1}$. Thus, given consecutive times $\tau_i^1< \tau_{i+1}^1$, one can always find a value of $j \in \{0,..., N_{\alpha/2^{p-1}, [s,t]}(\omega^2)\}$ such that $\tau_i^1 \leq \tau_j^2 \leq \tau^1_{i+1}$. To see this, observe that taking $k_i = \sup\{ \ell: \tau_\ell^2 \leq \tau_i^1 \}$, we have by the superadditivity of $\omega^1$ and $\omega^2$ that 
    \begin{equation*}
        \omega^2(\tau_{k_i}^2, \tau_{i+1}^1) \geq \omega^2(\tau_{k_i}^2, \tau_i^1) + \omega^2(\tau_i^1, \tau_{i+1}^1) \geq \omega^2(\tau_{k_i}^2, \tau_i^1) + \frac{\alpha}{2^{p-1}}\,.
    \end{equation*}
    It follows that setting $\tau_i^1 \leq \tau_{k_i+1}^2 \leq \tau_{i+1}^1$, and also that $i \leq k_i$. Thus we have 
    \begin{equation*}
        \tau_{i}^2 \leq \tau_{k_i}^2 \leq \tau_{i}^1\,,
    \end{equation*}
    for $i \in \{0, ..., N_{\alpha, [s,t]}(\omega^1)\}$. Finally we have
    \begin{align*}
        N_{\alpha, [s,t]}(\omega^1) &=\sup\{n \in \mathbb{N}\cup \{0\}: \tau^1_n(\alpha) < t\} \\
        &\leq  \sup\{n \in \mathbb{N}\cup \{0\}: \tau^2_n(\alpha/2^{p-1}) < t\} \\
        &= N_{\frac{\alpha}{2^{p-1}},[s,t]}(\omega^2)\,,
    \end{align*}
    which completes the proof.
\end{proof}

Next, we investigate how the strong convergence of $\mathbf{X}^\lambda \rightarrow \mathbf{X}$ in the inhomogeneous $p$-rough path metric affects the tails of $N_{\alpha,[s,t]}\big(\Vert \mathbf{X}^\lambda\Vert_{p-var}^p \big)$ and $N_{\alpha,[s,t]}\big(\Vert \mathbf{X} - \mathbf{X}^\lambda \Vert_{p-var}^p\big)$.
\begin{lemma}\label{lem:greedy-num-remainder}
    Suppose that $\mathbf{X}^\lambda \rightarrow \mathbf{X}$ strongly in the inhomogeneous $p$-rough path metric $\rho_{p-var}$, with rate $\delta(\lambda)$. Then
    \begin{equation*}
        \Pb \left( N_{\alpha, [s,t]}\big(\Vert \mathbf{X} - \mathbf{X}^\lambda\Vert_{p-var}^p\big) >0  \right) = o(\lambda^{-r})\,,
    \end{equation*}
    for all $r > 0$.
\end{lemma}
\begin{proof}
    We note that $N_{\alpha, [s,t]}(\omega) >0$ if and only if there exists some $\tau < t$ such that $\omega(s,\tau) \geq \alpha$. Since $\omega$ is superadditive, this implies $\omega(s,t) \geq \alpha$. Thus, 
    \begin{align*}
        \Pb \left( N_{\alpha, [s,t]} \big( \Vert \mathbf{X} - \mathbf{X}^\lambda \Vert_{p-var}^p \big) >0 \right) &\leq \Pb\left( \Vert \mathbf{X} - \mathbf{X}^\lambda \Vert_{p-var}^p \geq \alpha \right).
    \end{align*}
   Bounding the homogeneous norm by
    \begin{align*}
        \Vert \mathbf{X} - \mathbf{X}^\lambda \Vert_{p-var}^p &= \Vert X - X^\lambda\Vert_{p-var}^p + \Vert \mathbb{X} - \mathbb{X}^\lambda\Vert_{p/2-var}^{p/2} \\
        &\leq 2 \max\left\{ \rho_{p-var}(\mathbf{X}, \mathbf{X}^\lambda)^p, \rho_{p-var}(\mathbf{X}, \mathbf{X}^\lambda)^{p/2} \right\},
    \end{align*}
    we see that 
    \begin{align*}
        \Pb \left( N_{\alpha, [s,t]} \big( \Vert \mathbf{X} - \mathbf{X}^\lambda \Vert_{p-var}^p \big) >0 \right) &\leq \Pb \left( \max\left\{ \rho_{p-var}(\mathbf{X}, \mathbf{X}^\lambda)^p, \rho_{p-var}(\mathbf{X}, \mathbf{X}^\lambda)^{p/2} \right\} \geq \frac{\alpha}{2} \right) \\
        &\leq \Pb \left( \rho_{p-var}(\mathbf{X}, \mathbf{X}^\lambda) \geq \left( \frac{\alpha}{2}\right)^{1/p} \right) + \Pb \left( \rho_{p-var}(\mathbf{X}, \mathbf{X}^\lambda) \geq \left( \frac{\alpha}{2}\right)^{2/p} \right)\\
        &\leq 2 \Pb \left( \rho_{p-var}(\mathbf{X}, \mathbf{X}^\lambda) \geq \varepsilon \right),
    \end{align*}
   with $\varepsilon = \max\{ (\alpha/2)^{1/p}, (\alpha/2)^{2/p}\}$. Finally, we pick a $\lambda'$ such that $k \delta(\lambda) \leq \varepsilon$ for all $\lambda > \lambda'$. This yields 
   \begin{equation*}
    2 \Pb \left( \rho_{p-var}(\mathbf{X}, \mathbf{X}^\lambda) \geq \varepsilon \right) \leq 2 \Pb \left( \rho_{p-var}(\mathbf{X}, \mathbf{X}^\lambda ) \geq k\delta(\lambda) \right) = o(\lambda^{-r}),
   \end{equation*}
   as required.
\end{proof}

\begin{lemma}\label{lem:tail-greedy-approx}
    Let $\mathbf{X}^\lambda, \mathbf{X}$ be Gaussian rough paths of finite mixed $(1, \rho)$ variation, and suppose that $\mathbf{X}^\lambda \rightarrow \mathbf{X}$ strongly in inhomogeneous $p$-rough path metric with rate $\delta(\lambda)$, and $\alpha>0$. Then there exists some constant $c$ such that
    \begin{equation*}
        \Pb \left( N_{\alpha,[s,t]}\big(\Vert \mathbf{X}^\lambda\Vert_{p-var}^p \big) \geq u \right) \leq \exp \left\{ - c_1 u^{2/q}\right\}  + o(\lambda^{-r}).
    \end{equation*}
\end{lemma}

\begin{proof}
By Lemma \ref{lem:tail-inclusion}, we have 
\begin{equation}\label{eq:tail-inclusion-intermediary}
    \Pb \left( N_{\alpha,[s,t]}\big(\Vert \mathbf{X}^\lambda\Vert_{p-var}^p \big) \geq u \right) \leq \Pb \left( N_{\alpha_p,[s,t]}\left(\Vert \mathbf{X}\Vert_{p-var}^p + \Vert \mathbf{X} -\mathbf{X}^\lambda\Vert_{p-var}^p \right) \geq u \right)\,.
\end{equation}
Lemma 5 of \cite{2013friz:improved-lipschitz-estimate} states that 
\begin{equation*}
    N_{\alpha_p,[s,t]}\left(\Vert \mathbf{X}\Vert_{p-var}^p + \Vert \mathbf{X} -\mathbf{X}^\lambda\Vert_{p-var}^p \right) \leq 2 N_{\alpha_p, [s,t]}\big(\Vert \mathbf{X}\Vert_{p-var}^p\big) + 2 N_{\alpha_p, [s,t]}\big(\Vert \mathbf{X}-\mathbf{X}^\lambda\Vert_{p-var}^p\big) + 2\,,
\end{equation*}
which when substituted into \eqref{eq:tail-inclusion-intermediary} yields 
\begin{align*}
    \Pb \big( N_{\alpha,[s,t]}&\big(\Vert \mathbf{X}^\lambda\Vert_{p-var}^p \big) \geq u \big) \leq \Pb\left( N_{\alpha_p, [s,t]}\big(\Vert \mathbf{X}\Vert_{p-var}^p\big) +  N_{\alpha_p, [s,t]}\big(\Vert \mathbf{X}-\mathbf{X}^\lambda\Vert_{p-var}^p\big) \geq \frac{u-2}{2} \right)\\
    &\leq \Pb\left( N_{\alpha_p, [s,t]}\big(\Vert \mathbf{X}\Vert_{p-var}^p\big) \geq \frac{u}{4}  \right) + \Pb \left(N_{\alpha_p, [s,t]}\big(\Vert \mathbf{X}-\mathbf{X}^\lambda\Vert_{p-var}^p\big) \geq \frac{u-4}{4}  \right)    \,.
\end{align*}
Applying the  Weibull  tails of $N_{\alpha,[s,t]}\big(\Vert \mathbf{X}\Vert_{p-var}^p \big)$ and Lemma \ref{lem:greedy-num-remainder}, we see that 
\begin{equation*}
    \Pb \left( N_{\alpha,[s,t]}\big(\Vert \mathbf{X}^\lambda\Vert_{p-var}^p \big) \geq u \right) \leq \exp \left\{ - c_1 u^{2/q} \right\}  + o(\lambda^{-r}).
\end{equation*}

\end{proof}

\section{Lipschitz estimates for rough RSSDE}\label{sec:lipschitz}
In this section, we extend the (local) Lipschitz estimates for rough differential equations to the regime-switching case.

\begin{theorem}\label{thm:sup-estimate}
    Let $\mathbf{X}$ be a Gaussian rough path, $\{\mathbf{X}^\lambda\}_{\lambda >0}$ a family of rough path lifts of finite-variation processes, and $J$ be a Markov process independent of $\mathbf{X}$, $\mathbf{X}^\lambda$ with state space $\mathcal{E}$ and a.s. finite jump activity on compact intervals. Let $\gamma > p > 2$, and $\{V_i \}_{i \in \mathcal{E}}$ be a family of vector fields such that $|V_i|_{Lip^\gamma}$ is uniformly bounded by some $\nu$, that is, $\sup_{i \in \mathcal{E}}|V_i|_{Lip^\gamma} \leq \nu$. 
    
    Now, let $\{0=t_0^J<t_1^J<\cdots<t_{N^J+1}^J = T\}$ be a partition of $[0,T]$ given by the jump times of $J$. Define $Y^k$ and $Y^{k,\lambda}$ to be the RDE solutions to 
    \begin{equation*}
        \di Y^k_t = V_{J_{t_k}}(Y^k_t) \, \di \mathbf{X}_t\,, \quad Y^k_{t_k}= Y^{k-1}_{t_k}\,, \quad t \in [t_k, t_{k+1}],
    \end{equation*}
    and 
    \begin{equation*}
        \di Y^{k,\lambda}_t = V_{J_{t_k}}(Y^{k,\lambda}_t) \, \di \mathbf{X}^\lambda_t\,, \quad Y^{k,\lambda}_{t_k}= Y^{k-1, \lambda}_{t_k}\,, \quad t \in [t_k, t_{k+1}],
    \end{equation*}
    respectively, and define $Y$ (resp. $Y^\lambda$) to be the concatenation of the $\{Y^k\}_{k=0}^{N^J}$ (resp. $\{Y^{k, \lambda}\}_{k=0}^{N^J}$). Then, there exists some constant $C= C(\gamma, p, \nu, \alpha)$ such that
    \begin{equation*}
        \Vert Y - Y^\lambda\Vert_{\infty; [0,T]} \leq K_1 C^{N^J} \rho_{p-var;[0,T]}(\mathbf{X}, \mathbf{X}^\lambda) \exp\left\{ C \left( N_{\alpha, [0,T]}(\mathbf{X}) + N_{\alpha, [0,T]}(\mathbf{X}^\lambda) \right)  \right\},
    \end{equation*}
    where $K_1 = C^2/(C-1)$.

\end{theorem}
\begin{proof}
    Under the proposed conditions each $Y^{k}$ and $Y^{k,\lambda}$ exists uniquely and is continuous, and these properties extend to the concatenation $Y$ and $Y^\lambda$. Theorem 4 of \cite{friz2012integrability} yields the estimate
    \begin{align}
        \Vert Y^k - Y^{k, \lambda} \Vert_{\infty; [t_k, t_{k+1}]} &\leq C \left[ |Y_{t_k}^k - Y^{k, \lambda}_{t_k}| + \rho_{p-var;[t_k, t_{k+1}]}(\mathbf{X}, \mathbf{X}^\lambda) \right] \label{eq:theorem4-friz} \\
        & \qquad \times \exp\left\{C \big( N_{\alpha;[t_k, t_{k+1}]}(\mathbf{X}) + N_{\alpha;[t_k, t_{k+1}]}(\mathbf{X}^\lambda) \big) \right\} \nonumber
    \end{align}
    for each $k=0, \dots, J$. Noting that 
    \begin{align*}
        |Y_{t_k}^k - Y_{t_k}^{k, \lambda} | &\leq \Vert Y^k - Y^{k, \lambda} \Vert_{\infty; [t_{k-1}, t_k]} \\
        &\leq C \left[ |Y_{t_{k-1}}^{k-1} - Y^{k-1, \lambda}_{t_{k-1}}| + \rho_{p-var;[t_{k-1}, t_{k}]}(\mathbf{X}, \mathbf{X}^\lambda) \right] \\
        & \qquad \times \exp\left\{C \big( N_{\alpha;[t_{k-1}, t_{k}]}(\mathbf{X}) + N_{\alpha;[t_{k-1}, t_{k}]}(\mathbf{X}^\lambda) \big) \right\}\,,
    \end{align*}
    we can rewrite Equation \eqref{eq:theorem4-friz} as 
    \begin{align*}
        \Vert Y^k - Y^{k, \lambda}&\Vert_{\infty;[t_k, t_{k+1}]} \leq C \rho_{p-var;[t_k, t_{k+1}]}(\mathbf{X}, \mathbf{X}^\lambda) \exp\left\{C \big( N_{\alpha;[t_k, t_{k+1}]}(\mathbf{X}) + N_{\alpha;[t_k, t_{k+1}]}(\mathbf{X}^\lambda) \big) \right\} \\
        &+ C^2 \left[ |Y_{t_{k-1}}^{k-1} - Y^{k-1, \lambda}_{t_{k-1}}| + \rho_{p-var;[t_{k-1}, t_{k}]}(\mathbf{X}, \mathbf{X}^\lambda) \right] \\
        &  \times\exp\left\{C \big( N_{\alpha;[t_{k-1}, t_{k}]}(\mathbf{X}) + N_{\alpha;[t_k, t_{k+1}]}(\mathbf{X}) +  N_{\alpha;[t_{k-1}, t_{k}]}(\mathbf{X}^\lambda) + N_{\alpha;[t_k, t_{k+1}]}(\mathbf{X}^\lambda) \big) \right\}\,.
    \end{align*}
    Iterating back to $[0, t_1]$ yields 
    \begin{align*}
        \Vert Y^k - Y^{k, \lambda}\Vert_{\infty;[t_k, t_{k+1}]} &\leq \sum_{j=0}^{k} C^{j+1} \rho_{p-var; [t_{k-j}, t_{k+1-j}]}(\mathbf{X}, \mathbf{X}^\lambda) \\
        &  \times \exp \left\{C \left( \sum_{i=0}^j N_{\alpha;[t_{k-i}, t_{k+1-i}]}(\mathbf{X}) + N_{\alpha;[t_{k-i}, t_{k+1-i}]}(\mathbf{X}^\lambda) \right) \right\}\,.
    \end{align*}
    By Lemma \ref{lem:subadditivity-greedy} and the fact that $\rho_{p-var;[u,v]}(\cdot, \cdot) \leq \rho_{p-var;[s,t]}(\cdot, \cdot)$ if $[u,v] \subseteq [s,t]$, we have
    \begin{equation*}
        \Vert Y^k - Y^{k, \lambda}\Vert_{\infty; [t_k, t_{k+1}]} \leq \left( \sum_{j=1}^{k+1} C^j \right) \rho_{p-var;[0,t_{k+1}]}(\mathbf{X}, \mathbf{X}^\lambda) \exp \left\{ C \left( N_{\alpha,[0, t_{k+1}]}(\mathbf{X}) + N_{\alpha, [0,t_{k+1}]}(\mathbf{X}^\lambda) \right) \right\}\,.
    \end{equation*}
    Finally, noting that 
    \begin{equation*}
        \Vert Y - Y^\lambda\Vert_{\infty; [0,T]} = \max_{k=0,...,J} \Vert Y^k - Y^{k,\lambda}\Vert_{\infty; [t_k, t_{k+1}]}\,,
    \end{equation*}
    we arrive at the estimate
    \begin{align*}
        \Vert Y - Y^\lambda\Vert_{\infty; [0,T]} &\leq  \left( \sum_{j=1}^{J+1} C^j \right) \rho_{p-var;[0,T]}(\mathbf{X}, \mathbf{X}^\lambda) \exp\left\{ C \left( N_{\alpha, [0,T]}(\mathbf{X}) + N_{\alpha, [0,T]}(\mathbf{X}^\lambda) \right)  \right\}\\
        &\leq K_1 K_2^{J} \rho_{p-var;[0,T]}(\mathbf{X}, \mathbf{X}^\lambda) \exp\left\{ C \left( N_{\alpha, [0,T]}(\mathbf{X}) + N_{\alpha, [0,T]}(\mathbf{X}^\lambda) \right)  \right\}
    \end{align*}
    with $K_1 = C^2/(C-1)$ and $K_2 = C$.
\end{proof}

\section{Strong convergence}\label{sec:strong-convergence}
To establish the strong convergence of $Y^\lambda \rightarrow Y$ in supremum norm, we apply the methodology in \cite{nguyen2021rate} of utilising the probabilistic properties of the local Lipschitz coefficient of the solution map (Theorem \ref{thm:sup-estimate}). As in \cite{nguyen2021rate}, we impose a constraint on the jump process $J$ to control the growth of the constant $K_2^J$ appearing in Theorem \ref{thm:sup-estimate}.

\begin{assumption}\label{asm:jump-tails}
There exists some $\gamma_0 >0$ such that $\mathbb{P}\big( N^J > j) = o \big( \mathrm{e}^{-j (\log j - \gamma_0 )}\big)$.
\end{assumption}
Assumption \ref{asm:jump-tails} clearly holds if $J$ is deterministic. We refer to Lemma 4.3 of \cite{nguyen2021rate} which implies that Assumption \ref{asm:jump-tails} holds whenever $J$ is of bounded jump intensity.

\begin{lemma}
    Let $J$ be a jump process satisfying Assumption \ref{asm:jump-tails}. Then the number of jumps $N^J$ has finite expectation.
\end{lemma}
\begin{proof}
    Using the tail sum formula for expectation, we see that 
    \begin{equation*}
        \E \left[N^J\right] = \sum_{j=0}^\infty \Pb (N^J > j).
    \end{equation*}
    Since $\Pb(N^J > j) = o(\mathrm{e}^{-j (\log j - \gamma_0)})$, for large enough $j_0$ we have
    \begin{equation*}
        \E[ N^J] = \sum_{j= 0}^{j_0-1}  \Pb (N^J > j)  + \sum_{j=j_0}^\infty \mathrm{e}^{-j( \log j - \gamma_0)}
            \leq C_{j_0} + k \sum_{j=j_0}^\infty  \mathrm{e}^{-j/2} < \infty.
    \end{equation*}
\end{proof}

We are now in a position to show that $Y^\lambda \rightarrow Y$ strongly in supremum norm with rate slightly worse than $\mathbf{B}^\lambda \rightarrow \mathbf{B}$ in rough path metric.

\begin{theorem}\label{thm:strong-convergence} Under the conditions of Theorem \ref{thm:sup-estimate} and Assumption \ref{asm:jump-tails}, suppose that there exists $\delta: \mathbb{R}^+ \rightarrow \mathbb{R}^+$ with $\lim_{\lambda \rightarrow \infty} \delta(\lambda) = 0$ such that for all $r>0$ 
\begin{equation*}
    \mathbb{P}\big( \rho_{p-var;[0,T]}(\mathbf{B}^\lambda, \mathbf{B}) \geq k\delta(\lambda)\big) = o(\lambda^{-r})\,,
\end{equation*}
where $k=k(r,T)>0$ is a constant dependent on $r$ and $T$ only. Then there exists some constant $\beta =\beta(r,T)$ such that for all $\varepsilon >0$ 
\begin{equation}
    \mathbb{P}\big( \Vert  Y- Y^\lambda\Vert_{\infty; [0,T]} \geq \beta \delta(\lambda) \lambda^\varepsilon \big) = o(\lambda^{-r}).
\end{equation}
\end{theorem}
\begin{proof}
    Using the estimate in Theorem \ref{thm:sup-estimate}, we see that 
    \begin{align*}
        \Pb \left( \Vert Y- Y^\lambda\Vert_{\infty;[0,T]} \geq \beta \delta(\lambda) \lambda^{\varepsilon}  \right) &\leq \Pb \bigg(K_1 K_2^J \rho_{p-var;[0,T]}(\mathbf{B}, \mathbf{B}^\lambda) \\
        & \qquad \quad \times \exp\left\{ C \left( N_{\alpha, [0,T]}(\mathbf{B}) + N_{\alpha, [0,T]}(\mathbf{B}^\lambda) \right)  \right\} \geq \beta \delta(\lambda) \lambda^\varepsilon \bigg) \,.
    \end{align*}
    Let $f_{\gamma_0, r}(x):= \exp\big(x ( \log x - \gamma_0)/r\big)$ and $A_\lambda := f^{-1}_{\gamma_0, r}(\lambda)$ as in \cite{nguyen2021rate}. Write 
    \begin{equation*}
        E_\lambda = \left\{ K_1 K_2^J \rho_{p-var;[0,T]}(\mathbf{B}, \mathbf{B}^\lambda)  \exp\left\{ C \left( N_{\alpha, [0,T]}(\mathbf{B}) + N_{\alpha, [0,T]}(\mathbf{B}^\lambda) \right)  \right\} \geq \beta \delta(\lambda) \lambda^\varepsilon \right\},
    \end{equation*}
    and introduce the events
    \begin{align*}
        F_1 &= \left\{ J \leq A_\lambda  \right\}, \\
        F_2 &= \left\{ N_{\alpha,[0,T]}(\mathbf{B}) \leq \sqrt{\frac{(r+\varepsilon) \log \lambda}{c_1 }} \right\},\\
        F_3 &= \left\{N_{\alpha, [0,T]}(\mathbf{B}^\lambda) \leq \sqrt{\frac{(r+\varepsilon) \log \lambda}{c_1 }}\right\}.
    \end{align*}
    As
    \begin{equation}\label{eq:event-inequality}
        \Pb(E_\lambda) \leq \Pb(E_\lambda\cap F_1 \cap F_2 \cap F_3) + \Pb(F_1^\mathrm{c}) + \Pb(F_2^\mathrm{c}) + \Pb( F_3^\mathrm{c})\,,
    \end{equation}
    we can show strong convergence by showing that every term on the RHS of \eqref{eq:event-inequality} is $o(\lambda^{-r})$.
    Lemma \ref{lem:tail-greedy-approx} shows that 
    \begin{align*}
        \Pb(F_3^\mathrm{c}) &= \Pb \left( N_{\alpha, [0,T]}(\mathbf{B}^\lambda) \geq \sqrt{\frac{(r+\varepsilon)\log \lambda}{c_1 }}\right)  \\
        &\leq \exp \left\{ - \frac{c_1  (r+\varepsilon)\log \lambda}{c_1 } \right\} + o(\lambda^{-r})\\
        &= \frac{1}{\lambda^{r+\varepsilon}} + o(\lambda^{-r}) \\
        &= o(\lambda^{-r})\,,
    \end{align*}
    while the Gaussian tails of $N_{\alpha, [0,T]}(\mathbf{B})$ yield
    \begin{align*}
        \Pb(F_2^\mathrm{c}) &= \Pb \left( N_{\alpha,[0,T]}(\mathbf{B}) \geq \sqrt{\frac{(r+\varepsilon) \log \lambda}{c_1 }} \right) \\
        &\leq \exp\left\{- \frac{c_1  (r+\varepsilon)\log \lambda}{c_1 } \right\} \\
        &= \frac{1}{\lambda^{r+\varepsilon}} = o(\lambda^{-r}).
    \end{align*}
    The proof of Theorem 4.6 of \cite{nguyen2021rate} shows $\mathbb{P}\big(F_1^{\mathrm{c}}\big) =  \mathbb{P} \big( N^J > A_\lambda \big) = o(\lambda^{-r})$. Thus, \eqref{eq:event-inequality} becomes
    \begin{align*}
        \Pb(E_\lambda) & \leq \Pb \bigg( K_1 K_2^{A_\lambda} \rho_{p-var[0,T]}(\mathbf{B}, \mathbf{B}^\lambda) \\
        & \qquad \exp\left\{ C \left( \sqrt{\frac{(r+\varepsilon) \log \lambda}{c_1 }}  + \sqrt{\frac{(r+\varepsilon) \log \lambda}{c_1 }}\right)  \right\} \geq \beta \delta(\lambda) \lambda^{\varepsilon} \bigg) + o(\lambda^{-r})\,.
    \end{align*}
    Collecting constants that are independent of $\lambda$, the above becomes 
    \begin{equation*}
        \Pb(E_\lambda) \leq \Pb \left(\widetilde{K} K_2^{A_\lambda} \exp\left\{ \sqrt{\log \lambda}\right\} \lambda^{-\varepsilon_2} \rho_{p-var;[0,T]}(\mathbf{B}, \mathbf{B}^\lambda) \geq \beta \delta(\lambda) \right) + o(\lambda^{-r}),
    \end{equation*}
    for some constant $\widetilde{K}$. We will now show that 
    \begin{equation*}
        K_2^{A_\lambda} \exp \left\{ \sqrt{ \log \lambda} \right\} \lambda^{-\varepsilon} \rightarrow 0.
    \end{equation*}
    We have that $K_2^{A_\lambda} \lambda^{-\varepsilon/2} \rightarrow 0$ from \cite{nguyen2021rate} via the asymptotic decomposition of Lambert--W functions. Next, 
    \begin{equation*}
        \exp\{ \sqrt{\log{\lambda}} \} \lambda^{-\varepsilon/2} = \exp \left\{ \sqrt{\log \lambda} - \frac{\varepsilon}{2} \log \lambda \right\} \rightarrow 0\,.
    \end{equation*}
    As a result, we may choose some $\lambda'$ such that $K_2^{A_\lambda} \exp \left\{ \sqrt{\log \lambda} \right\} \lambda^{-\varepsilon/2} \leq 1$, for all $\lambda > \lambda'$. Thus, for all $\lambda > \lambda'$, 
    \begin{align*}
        \Pb(E_\lambda) \leq \Pb(\tilde{K} \rho_{p-var;[0,T]} ( \mathbf{B}, \mathbf{B}^\lambda) \geq \beta \delta(\lambda))\,.
    \end{align*}
    Finally, setting $\beta:=k/ \tilde{K}$, we see that 
    \begin{equation*}
        \Pb(E_\lambda) \leq \Pb \big( \rho_{p-var;[0,T]}(\mathbf{B}, \mathbf{B}^\lambda) \geq k \delta(\lambda) \big) = o(\lambda^{-r})\,,
    \end{equation*}
    as required.
\end{proof}

\begin{theorem}
    Let $\mathbf{X}$ and $\{\mathbf{X}^\lambda \}_{\lambda \geq 0}$ be Gaussian rough paths, and $J$ a jump process independent of $\mathbf{X}$ and $\mathbf{X}^\lambda$ satisfying Assumption \ref{asm:jump-tails}. Suppose there exists $\delta : \R^+ \rightarrow \R^+$ with $\lim_{\lambda \rightarrow \infty} \delta(\lambda) =0$ such that for all $r>0$ 
    \begin{equation*}
        \Pb \left( \rho_{p-var;[0,T]}(\mathbf{X}^\lambda, \mathbf{X}) \geq k \delta(\lambda)\right) = o(\lambda^{-r}),
    \end{equation*}
    where $k=k(r,T)>0$ is a constant dependent on $r$ and $T$ only. Then there exists some constant $\beta = \beta(r,T)$ such that for all $\varepsilon >0$, 
    \begin{equation*}
        \Pb\left(\Vert Y - Y^\lambda \Vert_{\infty;[0,T]} \geq \beta \delta(\lambda) \lambda^{\varepsilon} \right) = o(\lambda^{-r})
    \end{equation*}
\end{theorem}
\begin{proof}
    The same proof for Theorem \ref{thm:strong-convergence} holds with the sets $F_2$ and $F_3$ replaced with 
    \begin{align*}
        F_2 &= \left\{N_{\alpha, [0,T]}(\mathbf{X}) \leq \left( \frac{(r+\varepsilon) \log \lambda}{c_1 } \right)^{q/2} \right\}, \\
        F_3 &= \left\{N_{\alpha, [0,T]}(\mathbf{X}^\lambda) \leq \left( \frac{(r+\varepsilon) \log \lambda}{c_1 } \right)^{q/2} \right\}.
    \end{align*}
\end{proof}

\section{Approximation schemes}
Having proved Theorem \ref{thm:strong-convergence}, it remains to provide approximations of  enhanced Gaussian processes  that converge strongly. The following lemma will prove useful in this pursuit.

\begin{lemma}\label{lem:markov}
    Let $(M, d)$ be a complete metric space, $X$ a $M$-valued random variable and $\{X^\lambda\}_{\lambda \geq 0 }$ a collection of $M$-valued random variables. If there exist constants $C=C(q)>0$ and $\eta >0$ such that $\Vert d(X,X^\lambda)\Vert_{L^q}\leq C \lambda^{-\eta}$ for all $q>q'$, then $X^\lambda \rightarrow X$ strongly with rate function $\delta(\lambda) = \lambda^{-\gamma}$, for $\gamma <\eta$.
\end{lemma}
\begin{proof}
    By Markov's inequality, we write 
    \begin{align*}
        \Pb \left( d(X, X^\lambda) \geq k \delta(\lambda)\right) &\leq \frac{\E[ d(X, X^\lambda)^q]}{k^q \delta(\lambda)^q} \\
        &\leq \frac{C^q \lambda^{-q \eta}}{k^q \lambda^{-q\gamma}}\\
        &= \left( \frac{C}{k}\right)^q \lambda^{q(\gamma - \eta)}.
    \end{align*}
    To show that the RHS is $o(\lambda^{-r})$, we simply set $q = \frac{r+\varepsilon}{\eta - \gamma}$ provided $q>q'$. 
\end{proof}
 The class of Gaussian rough path we work with are those introduced in \cite{2013friz:improved-lipschitz-estimate}:
\begin{condition}[Condition 10, \cite{2013friz:improved-lipschitz-estimate}]\label{con:gaussian}
    Let $X= (X^1, ..., X^d)$ be a centred, continuous Gaussian process with independent components. Assume that the covariance of every component has \emph{H\"older dominated} finite mixed $(1,\rho)$-variation for some $\rho \in [1,2)$ on $[0,T]^2$, that is, there exists $K< \infty$ such that, for $k=1,...,d$ and uniformly over $s<t$ in $[0,T]$,  
    \begin{equation*}
        \Vert X\Vert_{(1,\rho)-var; [s,t]}:= \sup_{(t_i), (t_j') \in \mathcal{D}([s,t])} \left( \sum_{t_j'} \left( \sum_{t_i} \left|\E \left[ X_{t_i, t_{i+1}}^k X_{t_j', t_{j+1}'}^k \right] \right| \right)^\rho \right)^{1/\rho} \leq K \left( |t-s|^{1/\rho} \right).
    \end{equation*}
\end{condition}
The notion of $\textit{mixed}$ variation is a more refined notion than $p$-variation. Specifically, we note that fractional Brownian motion satisfies condition \ref{con:gaussian} with $\rho = \frac{1}{2H}$. The class of approximations to these processes that we work with are those satisfying the following condition, as presented in \cite{2013friz:improved-lipschitz-estimate}.
\begin{condition}\label{con:approx}
    Let $X$ be a Gaussian path satisfying condition \ref{con:gaussian}. Let $\{X^h\}_{h \in (0,1]}$ be a collection of centred, Gaussian processes with independent components for every $h \in (0,1]$ such that 
    \begin{enumerate}
        \item $(X^h, X):[0,T]\rightarrow \R^{2d}$ is jointly Gaussian, $(X^{h;i}, X^i)$ and $(X^{h;j},X^j)$ are independent for $i \neq j$, and 
        \begin{equation*}
            \left\vert\left\vert (X^h, X) \right\vert\right\vert_{(1,\rho)-var;[0,T]} =: K < \infty, \qquad \rho \in [1,2)
        \end{equation*}
        as in Condition \ref{con:gaussian}.
        \item The second moments converge uniformly, so that 
        \begin{equation*}
            \sup_{t \in [0,T]} \E \left[ \big\vert X_t^h - X_t \big\vert^2 \right] =:\delta(h)^{1/\rho} \rightarrow 0 \, \text{ for }h \rightarrow 0.
        \end{equation*}
    \end{enumerate}
\end{condition}
Condition \ref{con:approx} implies the convergence of the rough path lift $\mathbf{X}^h \rightarrow \mathbf{X}$ in $p$-variation rough path metric for $p > 2\rho$. Further to this, one can show that taking $X^h$ to be the linear interpolation of $X$ with mesh-size at most $h$, that $\delta(h)= h$.

Approximating the rough path lift $\mathbf{B}^H$ of a fractional Brownian motion $B^H$ via linear interpolation, Friz and Riedel \cite{friz2014gaussian-rates} provide the estimate 
\begin{equation*}
    \left|\left| \rho_{p-var;[0,T]}\big(\mathbf{B}^{H; n}, \mathbf{B}^H\big) \right|\right|_{L^r}  \leq C n^{-\eta}
\end{equation*}
with $\eta < H$. Applying Lemma \ref{lem:markov}, we see that $\mathbf{B}^n \rightarrow \mathbf{B}$ strongly with rate $\delta(n) = n^{-\gamma}$, for all $\gamma < H$ and $H\in \big( \frac{1}{4}, \frac{1}{2} \big]$.

%
%

\section*{Acknowledgments}
      GTN and OP gratefully acknowledge the support of the Australian Research Council DP180103106 grant. Additionally, OP acknowledges financial support from the Swiss National Science Foundation Project 200021-191984.

\bibliographystyle{alpha} 
\bibliography{ref}       

\begin{thebibliography}{BDTY20}

\bibitem[BDTY20]{bensoussan2020}
Alain Bensoussan, Boualem Djehiche, Hamidou Tembine, and Sheung Chi~Phillip
  Yam.
\newblock {Mean-Field-Type Games with Jump and Regime Switching}.
\newblock {\em Dynamic Games and Applications}, 10(1):19--57, 2020.

\bibitem[CFRS16]{2013friz:improved-lipschitz-estimate}
B.~Christian, P.~Friz, S.~Riedel, and J.~Schoenmakers.
\newblock From rough path estimates to multilevel monte carlo.
\newblock {\em SIAM Journal on Numerical Analysis}, 54(3):1449--1483, 2016.

\bibitem[CLL13]{2013cass:gaussian-tail-rde}
T.~Cass, C.~Litterer, and T.~Lyons.
\newblock {Integrability and tail estimates for Gaussian rough differential
  equations}.
\newblock {\em The Annals of Probability}, 41(4):3026 -- 3050, 2013.

\bibitem[DPR02]{duan2002option-pricing}
Jin-Chuan Duan, Ivilina Popova, and Peter Ritchken.
\newblock Option pricing under regime switching.
\newblock {\em Quantitative Finance}, 2(2):116, apr 2002.

\bibitem[FR13]{friz2012integrability}
P.~Friz and S.~Riedel.
\newblock Integrability of (non-)linear rough differential equations and
  integrals.
\newblock {\em Stochastic Analysis and Applications}, 31:336--358, 2013.

\bibitem[FR14]{friz2014gaussian-rates}
P.~Friz and S.~Riedel.
\newblock Convergence rates for gaussian rough paths.
\newblock 2014.

\bibitem[FV10]{friz2010introduction}
Peter~K. Friz and Nicolas~B. Victoir.
\newblock {\em Multidimensional Stochastic Processes as Rough Paths: Theory and
  Applications}.
\newblock Cambridge Studies in Advanced Mathematics. Cambridge University
  Press, 2010.

\bibitem[HF14]{friz2014rough-path-book}
Martin Hairer and Peter Friz.
\newblock {\em Rough path theory: with an introduciton to regularity
  structures}.
\newblock Springer, 2014.

\bibitem[IW92]{ikeda1992sde-diffusion}
N.~Ikeda and S.~Watanabe.
\newblock {\em Stochastic differential equations and diffusion processes}.
\newblock 1992.

\bibitem[Lyo98]{lyons1998rough_paths}
T.~Lyons.
\newblock {Differential equations driven by rough signals}.
\newblock {\em {Rev. Mat. Iberoam.}}, 14(2):215 -- 310, 1998.

\bibitem[NP21]{nguyen2021rate}
G.T. Nguyen and O.~Peralta.
\newblock Rate of strong convergence to solutions of regime-switching
  stochastic differential equations.
\newblock 2021.

\bibitem[RW85]{romisch1985}
W.~R\"omisch and A.~Wakolbinger.
\newblock {On Lipschitz dependence in systems with differentiated inputs}.
\newblock {\em Mathematische Annalen}, 272:237--248, 1985.

\bibitem[You36]{young1936holder}
L.~Young.
\newblock An inequality of h\"older type, connected with stieltjes integration.
\newblock {\em Acta Mathematica}, 67(none):251 -- 282, 1936.

\bibitem[YZ10]{yin2010hybrid_switching}
G.~Yin and C.~Zhu.
\newblock {\em Hybrid switching diffusions}.
\newblock Springer New York, NY, 1 edition, 2010.

\bibitem[ZC13]{zou2013}
Wei Zou and Jiahua Chen.
\newblock A markov regime-switching model for crude-oil markets: Comparison of
  composite likelihood and full likelihood.
\newblock {\em The Canadian Journal of Statistics / La Revue Canadienne de
  Statistique}, 41(2):353--367, 2013.

\end{thebibliography}

\end{document}